\documentclass{amsart}
\usepackage{amsmath, amsfonts, amsthm, enumerate, mathrsfs}

\newtheorem{theorem}{Theorem}[section]
\newtheorem{corollary}[theorem]{Corollary}
\newtheorem{lemma}[theorem]{Lemma}
\newtheorem{proposition}[theorem]{Proposition}
\newtheorem{observation}[theorem]{Observation}

\theoremstyle{definition}

\theoremstyle{remark}
\newtheorem*{remark}{Remark} 
\theoremstyle{definition}
\newtheorem{example}[theorem]{Example}

\numberwithin{equation}{section}

\begin{document}

\allowdisplaybreaks

\title[Diagonal elements in the Nonnegative Inverse Eigenvalue Problem]{Diagonal elements in the Nonnegative Inverse Eigenvalue Problem}


\author{Richard Ellard}
\address{
	Richard Ellard,
	School of Mathematics and Statistics,
	University College Dublin,
	Belfield, Dublin 4, Ireland
}
\email{richardellard@gmail.com}
\thanks{The authors' work was supported by Science Foundation Ireland under Grant 11/RFP.1/MTH/3157.}

\author{Helena \v{S}migoc}
\address{
	Helena \v{S}migoc,
	School of Mathematics and Statistics,
	University College Dublin,
	Belfield, Dublin 4, Ireland
}
\email{helena.smigoc@ucd.ie}



\date{February 2017}

\begin{abstract}
	We say that a list of complex numbers is \emph{realisable} if it is the spectrum of some (entrywise) nonnegative matrix. The Nonnegative Inverse Eigenvalue Problem (NIEP) is the problem of characterising all realisable lists. Although the NIEP remains unsolved, it has been solved in the case where every entry in the list (apart from the Perron eigenvalue) has nonpositive real part. For a given spectrum of this type, we show that a list of nonnegative numbers may arise as the diagonal elements of the realising matrix if and only if these numbers satisfy a remarkably simple inequality. Furthermore, we show that realisation can be achieved by the sum of a companion matrix and a diagonal matrix.
\end{abstract}

\maketitle

\section{Introduction}

The \emph{Nonnegative Inverse Eigenvalue Problem} (NIEP) is the problem of finding necessary and sufficient conditions which determine whether a given list of $n$ complex numbers can arise as the spectrum of some $n\times n$ (entrywise) nonnegative matrix. We say a list $\sigma:=(\lambda_1,\lambda_2,\ldots,\lambda_n)$ is \emph{realisable} if there exists a nonnegative matrix $A$ with spectrum $\sigma$.

Below, we state some well-known conditions which are necessary for a list to be realisable. Such conditions generally involve the power sums
\[
	s_k(\sigma):=\sum_{i=1}^n\lambda_i^k.
\]

\begin{theorem}[\bf Necessary conditions in the NIEP]\label{thm:NecessaryConditions}
	Suppose $\sigma:=(\lambda_1,\lambda_2,\ldots,$ $\lambda_n)$ is the spectrum of a  nonnegative
matrix $A$. Then
	\begin{enumerate}
		\item[\textup{(i)}] $\sigma$ is self-conjugate, i.e. $\overline{\sigma}:=\left(\overline{\lambda}_1,\overline{\lambda}_2,\ldots,\overline{\lambda}_n\right)=\sigma;$
		\item[\textup{(ii)}] $\max_i|\lambda_i|\in\sigma;$
		\item[\textup{(iii)}] $s_k(\sigma)\geq0$ for every positive integer $k$;
		\item[\textup{(iv)}] $s_k(\sigma)^m\leq n^{m-1}s_{km}(\sigma)$ for all positive integers $k$ and $m$.\footnote{These are known as the \emph{JLL Conditions} (see \cite{LoewyLondon, Johnson}).}
	\end{enumerate}
\end{theorem}





For $n>4$, the NIEP remains unsolved; however, numerous authors have made significant progress by restricting their attention to lists (or to matrices) of a certain type. In 1949, Sule\v{i}manova \cite{Suleimanova1949} showed that if $\sigma$ is a real list, with one positive every and $n-1$ negative entries, then $s_1(\sigma)\geq0$ is both necessary and sufficient for the existence of a nonnegative matrix with spectrum $\sigma$. This result was later generalised to complex lists by Laffey and \v{S}migoc:

\begin{theorem}{\scshape\cite{LaffeySmigoc}}\label{thm:LS}
  Let $\sigma:=(\rho,\lambda_2,\lambda_3,\ldots,\lambda_n)$ be a list of self-conjugate complex numbers with $\rho\geq0$ and $\mathrm{Re}\,\lambda_i\leq0$: $i=2,3,\ldots,n$. Then $\sigma$ is the spectrum of a nonnegative
matrix if and only if
\begin{equation}\label{eq:SumPositive}
	s_1(\sigma)\geq0
\end{equation}
and
\begin{equation}\label{eq:1stJLL}
	s_1(\sigma)^2\leq ns_2(\sigma).
\end{equation}
  Furthermore, when the above conditions are satisified, $\sigma$ may be realised by a matrix of the form
$C+\alpha I_n$, where $C$ is a nonnegative companion matrix with trace zero and $\alpha$ is a
nonnegative scalar.
\end{theorem}

If a realising matrix is known to exist for a given spectrum, one can consider the properties such a matrix might possess. In particular, in this paper, we consider the possible diagonal elements of such a matrix. Specifically, we show that if $\sigma$ contains one positive element and $n-1$ elements with negative real part, then $\sigma$ is realisable by a nonnegative matrix with diagonal elements $\Delta:=(a_1,a_2,\ldots,a_n)$ if and only if
\begin{equation*}
	s_1(\Delta)=s_1(\sigma)
\end{equation*}
and
\begin{equation*}
	s_2(\Delta)\leq s_2(\sigma).
\end{equation*}

Our work is motivated by some earlier constructive methods in the NIEP which rely heavily on diagonal elements (see \cite{SmigocDiagonalElement, SmigocSubmatrixConstruction}). To see how our results may be applied, consider, for example, the following theorem of \v{Smigoc}:

\begin{theorem}\textup{\bf\cite{SmigocDiagonalElement}\,}
	Let $(\rho,\lambda_2,\lambda_3,\ldots,\lambda_m)$ be realisable, where $\rho$ is the Perron
eigenvalue, and let $(\mu_1,\mu_2,\ldots,\mu_n)$ be the spectrum of a nonnegative matrix with a
diagonal element greater than or equal to $\rho$. Then
	\[
		(\mu_1,\mu_2,\ldots,\mu_n,\lambda_2,\lambda_3,\ldots,\lambda_m)
	\]
	is realisable.
\end{theorem}



\section{Statement of main result}

Below, we state the main result of this paper:

\begin{theorem}\label{thm:DiagonalElements}
	Let $\Delta:=(a_1,a_2,\ldots,a_n)$, where $a_1\geq a_2\geq\cdots\geq a_n\geq0$, let $\rho\geq0$ and let $(\lambda_2,\lambda_3,\ldots,\lambda_n)$ be a self-conjugate list of complex numbers with nonpositive real parts. Then the list $\sigma:=(\rho,\lambda_2,\lambda_3,\ldots,\lambda_n)$ is the spectrum of a nonnegative matrix with diagonal elements $\Delta$ if and only if
	\begin{equation}\label{eq:s1ACondition}
		s_1(\Delta)=s_1(\sigma)
	\end{equation}
	and
	\begin{equation}\label{eq:s2ACondition}
		s_2(\Delta)\leq s_2(\sigma).
	\end{equation}
	
Furthermore, if (\ref{eq:s1ACondition}) and (\ref{eq:s2ACondition}) are satisfied, then $\sigma$ may be realised by a nonnegative matrix of the form
\begin{equation}\label{eq:AMatrix}
	A:=\left[
		\begin{array}{ccccc}
			a_1 & 1 &  &  &  \\
			& a_2 & 1 &  &  \\
			&  & \ddots & \ddots &  \\
			&  &  & a_{n-1} & 1 \\
			b_n & b_{n-1} & \cdots & b_2 & a_n \\
		\end{array}
	\right].
\end{equation}
\end{theorem}

Observe that, by the Cauchy-Schwarz inequality, if a list $\Delta$ satisfying (\ref{eq:s1ACondition}) and (\ref{eq:s2ACondition}) exists, then (\ref{eq:SumPositive}) and (\ref{eq:1stJLL}) must hold. Furthermore, note that if $a_1=a_2=\cdots=a_n=s_1(\sigma)/n$ in Theorem \ref{thm:DiagonalElements}, then (\ref{eq:s2ACondition}) becomes the JLL condition (\ref{eq:1stJLL}) and the matrix $A$ takes the form $A=C+\frac{s_1(\sigma)}{n}I_n$, where $C$ is a companion matrix with trace zero. Hence Theorem \ref{thm:DiagonalElements} may be seen as a generalisation of Theorem \ref{thm:LS}.

Note that the matrix $A$, given in (\ref{eq:AMatrix}) is the sum of a companion matrix and a diagonal matrix. Let us begin by making some general observations about matrices of this form. First, we compute the characteristic polynomial of $A$:

\begin{lemma}\label{lem:ACharPol}
	Let $A$ be defined as in (\ref{eq:AMatrix}). Then the characteristic polynomial of $A$ is given by
	\begin{equation}\label{eq:ACharPol}
		f(x):=\prod_{i=1}^n(x-a_i)-b_2\prod_{i=1}^{n-2}(x-a_i)-b_3\prod_{i=1}^{n-3}(x-a_i)-\cdots-b_n.
	\end{equation}
\end{lemma}
\begin{proof}
	The fact that (\ref{eq:ACharPol}) is the characteristic polynomial of $A$ can be easily verified by computing the determinant of $xI_n-A$ by Laplace expansion along the first column and then using induction on $n$. 
\end{proof}

The entry $b_2$ has a special significance: if $A$ has spectrum $\sigma$, then, applying Newton's identities to the coefficient of $x^{n-2}$ in (\ref{eq:ACharPol}), we see that
\begin{align}\label{eq:b2}
	b_2&=\frac{1}{2}\left( s_1(\Delta)^2-s_2(\Delta) \right)-\frac{1}{2}\left( s_1(\sigma)^2-s_2(\sigma) \right)\notag\\
	&=\frac{1}{2}\left( s_2(\sigma)-s_2(\Delta) \right).
\end{align}
Hence, condition (\ref{eq:s2ACondition}) is directly related to the nonnegativity of $b_2$.

In general, it is easy to see that conditions (\ref{eq:s1ACondition}) and (\ref{eq:s2ACondition}) are necessary for the existence of a nonnegative matrix with spectrum $\sigma$ and diagonal elements $\Delta$:
\begin{observation}\label{obv:DiagonalElementsNecessity}
	Let $\sigma:=(\lambda_1,\lambda_2,\ldots,\lambda_n)$ be a list of complex numbers and let $\Delta:=(a_1,a_2,\ldots,a_n)$ be a list of nonnegative numbers. If $\sigma$ is the spectrum of a nonnegative matrix with diagonal elements $\Delta$, then
	\[
		s_1(\Delta)=s_1(\sigma)
	\]
	and
	\[
		s_m(\Delta)\leq s_m(\sigma) \hspace{2mm} : \hspace{6mm} m=2,3,\ldots
	\]
\end{observation}

Observation \ref{obv:DiagonalElementsNecessity} follows from the simple fact that $(A^m)_{ii}\geq a_i^m$ for each $i$. Note that, assuming $s_1(\Delta)=s_1(\sigma)$, we have
\[
	\sum_{i=1}^n\left( a_i-\frac{s_1(\sigma)}{n} \right)^2=s_2(\Delta)-\frac{s_1(\sigma)^2}{n}.
\]
In this case, (\ref{eq:s2ACondition}) is equivalent to
\begin{equation}\label{eq:CloseWRTl2}
	\sum_{i=1}^n\left( a_i-\frac{s_1(\sigma)}{n} \right)^2\leq s_2(\sigma)-\frac{s_1(\sigma)^2}{n},
\end{equation}
i.e. condition (\ref{eq:s2ACondition}) says that $(a_1,a_2,\ldots,a_n)$ must be sufficiently close to $(s_1(\sigma)/n,$ $s_1(\sigma)/n,\ldots,s_1(\sigma)/n)$ with respect to the $\ell^2$ norm.

\section{Small $n$}

Suppose $\lambda_1\geq\lambda_2$ and $a_1\geq a_2$. The general $2\times2$ matrix
\[
\left[
\begin{array}{cc}
 a_1 & a_{12} \\
 a_{21} & a_2
\end{array}
\right]
\]
has spectrum $(a_1+t,a_2-t)$ if and only if
\[
	a_{12}a_{21}=t(a_1-a_2+t).
\]
Since we require $a_{12}$ and $a_{21}$ to be nonnegative, it is not difficult to see that $(\lambda_1,\lambda_2)$ is the spectrum of a nonnegative matrix with diagonal elements $(a_1,a_2)$ if and only if $a_1\leq\lambda_1$ and $a_1+a_2=\lambda_1+\lambda_2$. If these conditions are satisfied, then, for example, choosing $a_{12}=1$ yields a matrix of the form (\ref{eq:AMatrix}).

For $n=3$, it is also relatively easy to show that (\ref{eq:s1ACondition}) and (\ref{eq:s2ACondition}) are essentially the required conditions, and some equivalent conditions for this case have recently been given by Hwang and Kim \cite{HwangKim}. For completeness, we will show that, under the appropriate conditions, the given spectrum can be realised by a nonnegative matrix of the form (\ref{eq:AMatrix}).

It is useful to distinguish the real and complex cases. Let us first consider the case when $\lambda_1$, $\lambda_2$ and $\lambda_3$ are real.

\begin{proposition}\label{prop:n=3RealDiagonalElements}
	If $\lambda_1\geq\lambda_2\geq\lambda_3$ and $a_1\geq a_2\geq a_3\geq0$, then the list $\sigma:=(\lambda_1,\lambda_2,\lambda_3)$ is the spectrum of a nonnegative matrix with diagonal elements $\Delta:=(a_1,a_2,a_3)$ if and only if the following conditions hold:
	\begin{enumerate}
		\item[\textup{(i)}] $\lambda_2\leq a_1\leq\lambda_1$;
		\item[\textup{(ii)}] $s_1(\Delta)=s_1(\sigma)$;
		\item[\textup{(iii)}] $s_2(\Delta)\leq s_2(\sigma)$.
	\end{enumerate}
	
	Furthermore, if (i)--(iii) are satisfied, then $(\lambda_1,\lambda_2,\lambda_3)$ may be realised by a nonnegative matrix of the form (\ref{eq:AMatrix}).
\end{proposition}
\begin{proof}
	The necessity of (ii) and (iii) was shown in Observation \ref{obv:DiagonalElementsNecessity}. Now suppose the matrix
	\[
A:=\left[
\begin{array}{ccc}
 a_1 & a_{12} & a_{13} \\
 a_{21} & a_2 & a_{23} \\
 a_{31} & a_{32} & a_3
\end{array}
\right]
	\]
	has spectrum $(\lambda_1,\lambda_2,\lambda_3)$. The characteristic polynomial of $A$ is
	\begin{align}\label{eq:3x3CharPol}
		(x-\lambda_1)(x-\lambda_2)&(x-\lambda_3)=\left(x-a_1\right)\left(x-a_2\right)\left(x-a_3\right)\\
		&-a_{23}a_{32}\left(x-a_1\right)-a_{13}a_{31}\left(x-a_2\right)\notag\\
		&-a_{12}a_{21}\left(x-a_3\right)-a_{12}a_{23}a_{31}-a_{21}a_{32}a_{13}.\notag
	\end{align}
	Letting $x=a_1$ in (\ref{eq:3x3CharPol}) gives
	\begin{align*}
		(a_1-\lambda_1)(a_1-\lambda_2)(a_1-\lambda_3)=&-a_{13}a_{31}\left(a_1-a_2\right)-a_{12}a_{21}\left(a_1-a_3\right)\\
		&-a_{12}a_{23}a_{31}-a_{21}a_{32}a_{13}\\
		\leq&\;0.
	\end{align*}
	Therefore, either $a_1\leq\lambda_3$ or $\lambda_2\leq a_1\leq\lambda_1$. Since $a_1+a_2+a_3=\lambda_1+\lambda_2+\lambda_3$, the former case would imply $a_1=a_2=a_3=\lambda_1=\lambda_2=\lambda_3$. Hence (i) holds.
	
	Now suppose (i)--(iii) hold and the matrix
	\begin{equation}\label{eq:3x3AMatrix}
		A:=\left[
			\begin{array}{ccc}
				a_1 & 1 & 0 \\
				0 & a_2 & 1 \\
				b_3 & b_2 & a_3
			\end{array}
		\right]
	\end{equation}
	has spectrum $(\lambda_1,\lambda_2,\lambda_3)$. By (\ref{eq:b2}), $b_2\geq0$. According to Lemma \ref{lem:ACharPol}, the characteristic polynomial of $A$ is
	\begin{equation}\label{eq:3x3CharPol2}
		(x-\lambda_1)(x-\lambda_2)(x-\lambda_3)=\left(x-a_1\right)\left(x-a_2\right)\left(x-a_3\right)-b_2(x-a_1)-b_3.
	\end{equation}
	Hence, letting $x=a_1$ in (\ref{eq:3x3CharPol2}), we see that
	\[
		b_3=-(a_1-\lambda_1)(a_1-\lambda_2)(a_1-\lambda_3),
	\]
	which, by (i), is nonnegative.
\end{proof}

Note that if we require the realising matrix to be symmetric, then the conditions on $(a_1,a_2,a_3)$ are more restrictive (see \cite{Fiedler}). For a survey and comparison of numerous results on the \emph{Symmetric Nonnegative Inverse Eigenvalue Problem}, the reader may be interested in \cite{ConnectingSufficientConditions}

\begin{example}\label{ex:3x3Real/SymmetricDifferent}
	The matrix
\[
\left[
\begin{array}{ccc}
 6 & 1 & 0 \\
 0 & 4 & 1 \\
 20 & 1 & 0
\end{array}
\right]
\]
	has spectrum $(\lambda_1,\lambda_2,\lambda_3)=(7,2,1)$ and diagonal elements $(a_1,a_2,a_3)$ $=(6,4,0)$; however, since $a_3<\lambda_3$, no symmetric nonnegative matrix exists with this spectrum and these diagonal elements (see \cite[Theorem 4.8]{Fiedler}).
\end{example}

Let us now consider the case when $\lambda_2$ and $\lambda_3$ are complex.

\begin{proposition}
	Let $\sigma:=(\rho,\alpha+i\beta,\alpha-i\beta)$, where $\rho,\beta\geq0$ and $\alpha$ is real, and let $\Delta:=(a_1,a_2,a_3)$, where $a_1\geq a_2\geq a_3\geq0$. Then $\sigma$ is the spectrum of a nonnegative matrix with diagonal elements $\Delta$ if and only if the following conditions hold:
	\begin{enumerate}
		\item[\textup{(i)}] $a_1\leq\rho$;
		\item[\textup{(ii)}] $s_1(\Delta)=s_1(\sigma)$;
		\item[\textup{(iii)}] $s_2(\Delta)\leq s_2(\sigma)$.
	\end{enumerate}
	
	Furthermore, if (i)--(iii) are satisfied, then $\sigma$ may be realised by a nonnegative matrix of the form (\ref{eq:AMatrix}).
\end{proposition}
\begin{proof}
	Suppose there exists a nonnegative matrix $A$ with spectrum $\sigma$ and diagonal elements $\Delta$. It is well known that the Perron eigenvalue of a nonnegative matrix must be at least as large as each of its entries. Hence (i) holds. The necessity of (ii) and (iii) was shown in Observation \ref{obv:DiagonalElementsNecessity}.
	
	Now suppose (i)--(iii) hold and the matrix $A$ given in (\ref{eq:3x3AMatrix}) has spectrum $\sigma$. As in the proof of Proposition \ref{prop:n=3RealDiagonalElements}, we note that $b_2\geq0$ by (\ref{eq:b2}) and
	\[
		b_3=-(a_1-\rho)\left( \left(a_1-\alpha\right){}^2+\beta ^2 \right)\geq0.\qedhere
	\]
\end{proof}

For larger $n$, we will require a systematic way to compute the $b_i$. This, we address in the following section.

\section{Computing the $b_i$}\label{sec:biComputation}

The observations in this section will apply to general (not necessarily nonnegative) matrices of the form (\ref{eq:AMatrix}).

Given $a_1,a_2,\ldots,a_n$ and an arbitrary polynomial
\begin{equation}\label{eq:ArbitraryPolynomial}
	p(x):=x^n-s_1(\Delta)x^{n-1}+q_2x^{n-2}+q_3x^{n-3}+\cdots+q_n,
\end{equation}
whose roots sum to $s_1(\Delta)$, there exists a unique solution $(b_2,b_3,\ldots,$ $b_n)$ such that the polynomials (\ref{eq:ACharPol}) and (\ref{eq:ArbitraryPolynomial}) coincide. The aim of this section is to find this solution in closed form.

The $k$-th elementary symmetric function of the variables $x_1,x_2,\ldots,x_n$ is defined by
\begin{align*}
	e_0(x_1,x_2,\ldots,x_n) &:= 1, \\
	e_k(x_1,x_2,\ldots,x_n) &:= \sum_{1\leq i_1<i_2<\cdots<i_k\leq n} x_{i_1}x_{i_2}\cdots x_{i_k} :\hspace{3mm} k=1,2,\ldots,n
\end{align*}
and it will also be convenient to define $e_k(x_1,x_2,\ldots,x_n)=0$ if $k<0$ or $k>n$.

Let us compare the coefficients of (\ref{eq:ACharPol}) and (\ref{eq:ArbitraryPolynomial}).  Both polynomials are monic and the coefficient of $x^{n-1}$ in both polynomials is $-s_1(\Delta)$. Comparing the coefficients of $x^{n-2}$ yields
\[
	b_2=-q_2+e_2(a_1,a_2,\ldots,a_n).
\]
Comparing the coefficients of $x^{n-3}$ gives
\[
	b_3=-q_3-e_3(a_1,a_2,\ldots,a_n)+e_1(a_1,a_2,\ldots,a_{n-2})b_2.
\]
In general, to compute the coefficient of $x^{n-k}$ in (\ref{eq:ACharPol}), we notice that the coefficient of $x^{n-k}$ in $\prod_{i=1}^n(x-a_i)$ is $(-1)^ke_k(\Delta)$, and, for $j=2,3,\ldots,k-1$, the coefficient of $x^{n-k}$ in $b_j\prod_{i=1}^{n-j}(x-a_i)$ is given by
\[
	(-1)^{k-j}e_{k-j}(\Delta^{(j)})b_j,
\]
where
\[
	\Delta^{(j)}:=(a_1,a_2,\ldots,a_j)  \hspace{2mm} : \hspace{6mm} j=1,2,\ldots,n.
\]
Hence, equating the coefficients of $x^{n-k}$ in (\ref{eq:ACharPol}) and (\ref{eq:ArbitraryPolynomial}) gives
\begin{equation}\label{eq:ExampleRecurrence}
	b_k=-q_k+(-1)^ke_k(\Delta)-\sum_{j=2}^{k-1}(-1)^{k-j}e_{k-j}(\Delta^{(n-j)})b_j.
\end{equation}
Using (\ref{eq:ExampleRecurrence}), it is possible to calculate the $b_i$ (recursively) by considering the elementary symmetric functions of the truncated lists $\Delta^{(j)}$.

For example, suppose we wish the matrix $A$, given in (\ref{eq:AMatrix}), to have spectrum $\sigma:=(3,2i,-2i,-1)$. Clearly, we must choose $(a_1,a_2,a_3,a_4)$ such that $a_1+a_2+a_3+a_4=2$. By equating the polynomials
\[
	f(x)=(x-a_1)(x-a_2)(x-a_3)(x-a_4)-b_2(x-a_1)(x-a_2)-b_3(x-a_1)-b_4
\]
and
\[
	p(x)=(x - 3) (x^2 + 4) (x + 1)
\]
and solving the recurrence relation (\ref{eq:ExampleRecurrence}), we see that we must set
\begin{align*}
	b_2&=-1+e_2(a_1,a_2,a_3,a_4)\\
	&= -1+a_1 a_2+a_1 a_3+a_2 a_3+a_1 a_4+a_2 a_4+a_3 a_4,\\
	b_3&=8-e_3(a_1,a_2,a_3,a_4)+e_1(a_1,a_2)b_2\\
	&= 8-a_1-a_2+a_1^2 a_2+a_1 a_2^2+a_1^2 a_3+a_1 a_2 a_3+a_2^2 a_3+a_1^2 a_4+a_1 a_2 a_4+a_2^2 a_4,\\
	b_4&=12+e_4(a_1,a_2,a_3,a_4)-e_2(a_1,a_2)b_2+e_1(a_1)b_3\\
	&= 12+8 a_1-a_1^2+a_1^3 a_2+a_1^3 a_3+a_1^3 a_4.
\end{align*}

As before, we note that (under the assumption $a_1+a_2+a_3+a_4=2$) the condition $b_2\geq0$ is equivalent to $a_1^2+a_2^2+a_3^2+a_4^2\leq s_2(\sigma)=2$. Hence, if Theorem \ref{thm:DiagonalElements} is to be believed, any choice of $(a_1,a_2,a_3,a_4)$ satisfying $b_2\geq0$ must also satisfy $b_3,b_4\geq0$. This example illustrates the difficulty in proving Theorem \ref{thm:DiagonalElements}.

To find a closed-form solution to (\ref{eq:ExampleRecurrence}), we will need to consider the $k$-th \emph{complete homogeneous symmetric functions}
\begin{align*}
	h_0(x_1,x_2,\ldots,x_n) &:= 1, \\
	h_k(x_1,x_2,\ldots,x_n) &:= \sum_{1\leq i_1\leq\cdots\leq i_k\leq n} x_{i_1}x_{i_2}\cdots x_{i_k} :\hspace{2mm} k=1,2,\ldots
\end{align*}
Again, it is convenient to define $h_k(x_1,x_2,\ldots,x_n)=0$ if $k<0$. 

The following relation between the elementary symmetric functions and the complete homogeneous ones is well known: 

\begin{lemma}\label{lem:ehRelation}
	For all $m>0$,
	\begin{equation}\label{eq:ehRelation}
		\sum_{i=0}^m(-1)^ie_i(x_1,x_2,\ldots,x_n)h_{m-i}(x_1,x_2,\ldots,x_n)=0.
	\end{equation}
\end{lemma}

We will need a generalisation of Lemma \ref{lem:ehRelation} to truncated lists. Given $a_1,a_2,\ldots,a_n$, let us formally define
\[
	\Delta^{(j)}:=\left\{
		\begin{array}{lcl}
			(a_1,a_2,\ldots,a_j) & : & j=1,2,\ldots,n,\\
			(a_1,a_2,\ldots,a_n,\underbrace{0,0,\ldots,0}_{j-n \text{ zeros}}) & : & j=n+1,n+2,\ldots
		\end{array}
	\right.
\]
To avoid unnecessarily long expressions, for $j=1,2,\ldots$, let us denote
\begin{align*}
	\epsilon_i^{(j)}&:=e_i\left(\Delta^{(j)}\right),\\
	\eta_i^{(j)}&:=h_i\left(\Delta^{(j)}\right)
\end{align*}
and let
\[
	\epsilon_i^{(0)}=\eta_i^{(0)}=\left\{
		\begin{array}{lcl}
			1 & : & i=0,\\
			0 & : & i\neq0.
		\end{array}
	\right.
\]
It will also be convenient to define
\[
	{\epsilon_i'}^{(j)}:=e_i(a_{j+1},a_{j+2},\ldots,a_n):\hspace{3mm} i=0,\ldots,n-j;\hspace{2mm} j=0,\ldots,n-1.
\]

\begin{observation}\label{obv:TruncatedListFacts}\
	\begin{enumerate}
		\item If $j\geq n$, then $\epsilon_i^{(j)}=\epsilon_i^{(n)}$ and $\eta_i^{(j)}=\eta_i^{(n)}$.
		\item For all $j,k\in\mathbb{N}$,
			\[
				\epsilon_j^{(k)}=\epsilon_j^{(k-1)}+a_k\epsilon_{j-1}^{(k-1)}
			\]
			and
			\[
				\eta_j^{(k)}=\sum_{r=0}^ja_k^r\eta_{j-r}^{(k-1)}.
			\]
	\end{enumerate}
\end{observation}

With the above notation, we give a generalisation of Lemma \ref{lem:ehRelation}, featuring truncated lists:

\begin{lemma}\label{lem:STRelation}
	If $k\geq0$ and $m>0$, then
	\begin{equation}\label{eq:STRelation}
		\sum_{i=0}^m(-1)^{i}\epsilon_i^{(k+i)}\eta_{m-i}^{(k+i+1)}=0.
	\end{equation}
\end{lemma}
\begin{proof}
	We first note that (\ref{eq:STRelation}) is independent of $a_{k+m+2},a_{k+m+3},\ldots$, and hence, we need only consider $n\leq k+m+1$. The proof is by induction on $n$.
	
	As a base of induction, we note that if $n\leq k+1$, then, by Observation \ref{obv:TruncatedListFacts}, the left-hand side of (\ref{eq:STRelation}) reduces to
	\[
		\sum_{i=0}^m(-1)^i\epsilon_i^{(n)}\eta_i^{(n)},
	\]
	which, by Lemma \ref{lem:ehRelation}, equals zero.
	
	Now suppose $n=k+l$, where $2\leq l\leq m$ and assume the statement holds for lists of length $n-1$. Note that, by the inductive hypothesis and Observation \ref{obv:TruncatedListFacts}, Item 1,
	\begin{equation}\label{eq:STExpansionInduction}
		\sum_{i=0}^{l-2}(-1)^i\epsilon_i^{(n-l+i)}\eta_{m-i}^{(n-l+i+1)}+\sum_{i=l-1}^m(-1)^i\epsilon_i^{(n-1)}\eta_{m-i}^{(n-1)}=0.
	\end{equation}
	
	The left-hand side of (\ref{eq:STRelation}) may be written as
	\begin{multline*}
		\sum_{i=0}^{l-2}(-1)^i\epsilon_i^{(n-l+i)}\eta_{m-i}^{(n-l+i+1)}\\+(-1)^{l-1}\epsilon_{l-1}^{(n-1)}\eta_{m-l+1}^{(n)}+\sum_{i=l}^m(-1)^i\epsilon_i^{(n)}\eta_{m-i}^{(n)},
	\end{multline*}
	which, by Observation \ref{obv:TruncatedListFacts}, Item 2, may be written in the form
	\begin{multline}\label{eq:STExpansion}
		\sum_{i=0}^{l-2}(-1)^i\epsilon_i^{(n-l+i)}\eta_{m-i}^{(n-l+i+1)}+(-1)^{l-1}\epsilon_{l-1}^{(n-1)}\sum_{r=0}^{m-l+1}a_n^r\eta_{m-l-r+1}^{(n-1)} \\ +\sum_{i=l}^m(-1)^i\left( \epsilon_i^{(n-1)}+a_n\epsilon_{i-1}^{(n-1)} \right)\sum_{r=0}^{m-i}a_n^r\eta_{m-i-r}^{(n-1)}.
	\end{multline}
	
	Let us now consider (\ref{eq:STExpansion}) as a polynomial in $a_n$. The constant term is equal to (\ref{eq:STExpansionInduction}), which, as previously noted, equals zero, and for $s=1,2,\ldots$, the coefficient of $a_n^s$ is given by
	\begin{align*}
		\sum_{i=l-1}^m(-&1)^i\epsilon_i^{(n-1)}\eta_{m-i-s}^{(n-1)}+\sum_{i=l}^m(-1)^i\epsilon_{i-1}^{(n-1)}\eta_{m-i-s+1}^{(n-1)}\\
		&=\sum_{i=l-1}^m(-1)^i\epsilon_i^{(n-1)}\eta_{m-i-s}^{(n-1)}-\sum_{i=l-1}^{m-1}(-1)^i\epsilon_i^{(n-1)}\eta_{m-i-s}^{(n-1)}\\
		&=(-1)^m\epsilon_m^{(n-1)}\eta_{-s}^{(n-1)}\\
		&=0.
	\end{align*}
	Hence (\ref{eq:STExpansion}) equals zero, as required.
	
	Finally, if $n=k+m+1$, then the left-hand side of (\ref{eq:STRelation}) may be written as
	\begin{equation}\label{eq:n=k+m+1}
		\sum_{i=0}^{m-1}(-1)^i\epsilon_i^{(n-m+i-1)}\eta_{m-i}^{(n-m+i)}+(-1)^m\epsilon_m^{(n-1)}.
	\end{equation}
	Since (\ref{eq:n=k+m+1}) is independent of $a_n$, it must vanish by the inductive hypothesis.
\end{proof}

\begin{remark}
	Clearly, Lemma \ref{lem:STRelation} holds for arbitrary (not necessarily real) variables $a_1,a_2,\ldots,a_n$.
\end{remark}

We are now able to give a closed form solution for each $b_i$:

\begin{lemma}\label{lem:bjSol}
	Let
	\[
		g(x):=x^{n-1}+c_1x^{n-2}+c_2x^{n-3}+\cdots+c_{n-1}
	\]
	be an arbitrary polynomial and let
	\begin{equation}\label{eq:ValueforRho'}
		\rho:=c_1+\epsilon_1^{(n)}.
	\end{equation}
	Then the matrix $A$ given in (\ref{eq:AMatrix}) has characteristic polynomial $(x-\rho)g(x)$ if and only if
	\begin{equation}\label{eq:bjSol}
		b_j=\sum_{i=0}^j\mathscr{K}_{i,j} \hspace{2mm} : \hspace{6mm} j=2,3,\ldots,n,
	\end{equation}
	where
	\begin{equation}\label{eq:Kij}
		\mathscr{K}_{i,j}:=\eta_i^{(n-j+1)}(c_1c_{j-i-1}-c_{j-i})+\epsilon_1^{(n)}\eta_{i-1}^{(n-j+1)}c_{j-i}.
	\end{equation}
\end{lemma}
\begin{proof}
	Equating the coefficients of $x^{n-j}$ in $(x-\rho)g(x)$ and (\ref{eq:ACharPol}) gives
	\begin{equation}\label{eq:bjRecurrence}
		c_j-\rho c_{j-1}=-\sum _{r=0}^j (-1)^r\epsilon_r^{(n-j+r)}b_{j-r}.
	\end{equation}
	Hence, in order for the polynomial in (\ref{eq:ACharPol}) to be equal to $(x-\rho)g(x)$, $(b_2,b_3,$ $\ldots,b_n)$ must be the unique solution to the recurrence relation (\ref{eq:bjRecurrence}) with initial condition $b_0=-1$. Therefore, we must show that the solution (\ref{eq:bjSol}) satisfies (\ref{eq:bjRecurrence}). For $b_j$ given as in (\ref{eq:bjSol}),
	\begin{equation}\label{eq:ijsrSub}
		\sum _{r=0}^j(-1)^r\epsilon_r^{(n-j+r)}b_{j-r}=\sum _{r=0}^j(-1)^r\epsilon_r^{(n-j+r)}\sum_{i=0}^{j-r}\mathscr{K}_{i,j-r}
	\end{equation}
	and hence, after substituting $i=j-s-r$ in (\ref{eq:ijsrSub}), we see that
	\begin{align*}
		\sum _{r=0}^j(-1)^r\epsilon_r^{(n-j+r)}b_{j-r}&=\sum _{r=0}^{j}\sum _{s=0}^{j-r}(-1)^r\epsilon_r^{(n-j+r)}\mathscr{K}_{j-s-r,j-r}\\
		&=\sum _{s=0}^j \sum _{r=0}^{j-s} (-1)^r\epsilon_r^{(n-j+r)}\mathscr{K}_{j-s-r,j-r}.
	\end{align*}
	From (\ref{eq:ValueforRho'}) and the definition (\ref{eq:Kij}) of $\mathscr{K}_{i,j}$, we have
	\begin{align*}
		c_j-\rho c_{j-1}&+\sum _{r=0}^j(-1)^r\epsilon_r^{(n-j+r)}b_{j-r}\\
		=\hspace{2mm}&-(c_1c_{j-1}-c_j)-\epsilon_1^{(n)}c_{j-1}\\
		&+\sum_{s=0}^{j}(c_1c_{s-1}-c_s)\sum_{r=0}^{j-s}(-1)^r\epsilon_r^{(n-j+r)}\eta_{j-s-r}^{(n-j+r+1)}\\
		&+\epsilon_1^{(n)}\sum_{s=0}^jc_s\sum_{r=0}^{j-s}(-1)^r\epsilon_r^{(n-j+r)}\eta_{j-s-r-1}^{(n-j+r+1)};
	\end{align*}
	however, we note that
	\begin{multline}\label{eq:VanishingExtraTerms}
		\sum_{s=0}^jc_s\sum_{r=0}^{j-s}(-1)^r\epsilon_r^{(n-j+r)}\eta_{j-s-r-1}^{(n-j+r+1)}=\\
		\sum_{s=0}^{j-1}c_s\sum_{r=0}^{j-s-1}(-1)^r\epsilon_r^{(n-j+r)}\eta_{j-s-r-1}^{(n-j+r+1)},
	\end{multline}
	since the additional terms on the left-hand side of (\ref{eq:VanishingExtraTerms}) vanish. Therefore
	\begin{align*}
		c_j-\rho c_{j-1}&+\sum _{r=0}^j(-1)^r\epsilon_r^{(n-j+r)}b_{j-r}\\
		=\hspace{2mm}&-(c_1c_{j-1}-c_j)-\epsilon_1^{(n)}c_{j-1}\\
		&+\sum_{s=0}^{j}(c_1c_{s-1}-c_s)\sum_{r=0}^{j-s}(-1)^r\epsilon_r^{(n-j+r)}\eta_{j-s-r}^{(n-j+r+1)}\\
		&+\epsilon_1^{(n)}\sum_{s=0}^{j-1}c_s\sum_{r=0}^{j-s-1}(-1)^r\epsilon_r^{(n-j+r)}\eta_{j-s-r-1}^{(n-j+r+1)}\\
		=\hspace{2mm}&\:\sum_{s=0}^{j-1}(c_1c_{s-1}-c_s)\sum_{r=0}^{j-s}(-1)^r\epsilon_r^{(n-j+r)}\eta_{j-s-r}^{(n-j+r+1)}\\
		&+\epsilon_1^{(n)}\sum_{s=0}^{j-2}c_s\sum_{r=0}^{j-s-1}(-1)^r\epsilon_r^{(n-j+r)}\eta_{j-s-r-1}^{(n-j+r+1)}\\
		=\hspace{2mm}&\:0,
	\end{align*}
	where the final equality follows from Lemma \ref{lem:STRelation}.
\end{proof}

\begin{remark}
	Note that, if the solution (\ref{eq:bjSol}) for $b_j$ is considered as a multivariable polynomial in $a_1,a_2,\ldots,a_n$, then $\mathscr{K}_{i,j}$ is the sum of all terms of degree $i$. For all $j=2,3,\ldots,n$ and $m=2,3,\ldots,j$, the sum of all terms of degree $m$ on the right-hand side of (\ref{eq:bjRecurrence}) is equal to
	\[
		-\sum_{r=0}^m(-1)^r\epsilon_r^{(n-j+r)}\mathscr{K}_{m-r,j-r},
	\]
	and hence, (\ref{eq:bjRecurrence}) implies
	\begin{equation}\label{eq:KijRecurrence}
		\sum_{r=0}^m(-1)^r\epsilon_r^{(n-j+r)}\mathscr{K}_{m-r,j-r}=0
	\end{equation}
	for all $j=2,3,\ldots,n$ and $m=2,3,\ldots,j$.
\end{remark}

\section{Proof of main result}\label{sec:ProofOfMainResult}

In the previous section, we considered the elementary symmetric functions of the real numbers $a_1,a_2,\ldots,a_n$, but we have yet to examine the elementary symmetric functions of the complex numbers $\rho,\lambda_2,\lambda_3,\ldots,\lambda_n$. Since $\lambda_2,\lambda_3,\ldots,\lambda_n$ shall have negative real parts, the signs of these elementary symmetric functions may be either positive or negative. It turns out that there is some regularity in the pattern of these signs, as was shown in \cite{HermiteBiehlerExtension}:

\begin{theorem}{\bf\cite{HermiteBiehlerExtension}}\label{thm:LSVariant}
	Let
	\[
		f(x):=(x-\rho)\prod_{i=2}^n(x-\lambda_i)=x^n+c_1x^{n-1}+c_2x^{n-2}+\cdots+c_n
	\]
	be a real polynomial, where $\rho$ is real and $\lambda_2,\lambda_3,\ldots,\lambda_n$ are complex numbers with nonpositive real parts. Then for each $k\in\{1,2,\ldots,n-2\}$, $c_k\leq0$ implies $c_{k+2}\leq0$.
\end{theorem}

In order to show that the $b_i$ are nonnegative under the hypotheses of Theorem \ref{thm:DiagonalElements}, we will require some inequalities involving elementary symmetric functions of complex variables. The required inequalities (which take a similar form to Newton's Inequalities) are given in \cite{NewtonLikeInequalities}.

\begin{theorem}{\bf\cite{NewtonLikeInequalities}}\label{thm:DifferentParity}
	Let $\mathcal{X}:=(x_1,x_2,\ldots,x_n)$ be a list of self-conjugate variables with nonnegative real parts. If $k$ and $l$ have different parity, $1\leq k<l\leq n-1$, then
	\begin{equation*}
		e_k(\mathcal{X})e_l(\mathcal{X})\geq e_{k-1}(\mathcal{X})e_{l+1}(\mathcal{X}).
	\end{equation*}
\end{theorem}

Since the real parts of the numbers $\rho,\lambda_2,\lambda_3,\ldots,\lambda_n$ do not all share the same sign, Theorem \ref{thm:DifferentParity} alone will not suffice.

Let $\mathcal{X}:=(x_1,x_2,\ldots,x_n)$ be a list of self-conjugate complex numbers with nonnegative real parts and let $r$ be real. Define
\begin{align}
	Q_{2k}(r;\mathcal{X})&:=\frac{e_{2k}(\mathcal{X})-re_{2k-1}(\mathcal{X})}{\binom{\lceil n/2 \rceil}{k}} \: : \hspace{4mm} k=0,1,\ldots,\lceil n/2 \rceil, \label{eq:QNotation} \\
	Q_{2k+1}(r;\mathcal{X})&:=\frac{e_{2k+1}(\mathcal{X})-re_{2k}(\mathcal{X})}{\binom{\lfloor n/2 \rfloor}{k}} \: : \hspace{4mm} k=0,1,\ldots,\lfloor n/2 \rfloor \notag
\end{align}
and observe that for even $n$ (say $n=2m$), if
\begin{equation}\label{eq:gForASLemma}
	g(x):=x^{2m}+e_1(\mathcal{X})x^{2m-1}+e_2(\mathcal{X})x^{2m-2}+\cdots+e_{2m}(\mathcal{X}),
\end{equation}
then the polynomial
\begin{align}
	(x-r)g(x)&=x^{2m+1}+Q_1(r;\mathcal{X})x^{2m}\notag\\
	&+\binom{m}{1}Q_2(r;\mathcal{X})x^{2m-1}+\binom{m}{1}Q_3(r;\mathcal{X})x^{2m-2}+\cdots\notag\\
	&+\binom{m}{m}Q_{2m}(r;\mathcal{X})x+\binom{m}{m}Q_{2m+1}(r;\mathcal{X})\label{eq:rXPolyEven}
\end{align}
has roots $r,-x_1,-x_2,\ldots,-x_n$.


\begin{lemma}\label{lem:ASInequalities}
	Let $\mathcal{X}:=(x_1,x_2,\ldots,x_n)$ be a self-conjugate list of complex numbers with nonnegative real parts and let $r\in\mathbb{R}$. If $Q_{2s}(r;\mathcal{X})>0$, then $Q_{2i}(r;\mathcal{X})>0:$ $i=0,1,\ldots,s$ and
	\begin{equation}\label{eq:ASInequalitiesEven}
		Q_{2k}(r;\mathcal{X})Q_{2l}(r;\mathcal{X})\geq Q_{2k-2}(r;\mathcal{X})Q_{2l+2}(r;\mathcal{X}): \hspace{4mm} 1\leq k\leq l\leq s-1.
	\end{equation}
	Similarly, if $Q_{2t+1}(r;\mathcal{X})>0$, then $Q_{2i+1}(r;\mathcal{X})>0:$ $i=0,1,\ldots,t$ and
	\begin{multline}\label{eq:ASInequalitiesOdd}
		Q_{2k+1}(r;\mathcal{X})Q_{2l+1}(r;\mathcal{X})\geq Q_{2k-1}(r;\mathcal{X})Q_{2l+3}(r;\mathcal{X}):\\ 1\leq k\leq l\leq t-1.
	\end{multline}
\end{lemma}
\begin{proof}
	Let $Q_{2s}(r;\mathcal{X})$ and $Q_{2t+1}(r;\mathcal{X})$ be positive. First suppose $n$ is even and write $n=2m$. Defining $g$ as in (\ref{eq:gForASLemma}), recall that the polynomial (\ref{eq:rXPolyEven}) has roots $r,-x_1,-x_2,\ldots,-x_n$, and therefore, by Theorem \ref{thm:LSVariant}, $Q_{2i}(r;\mathcal{X})>0$ for each $i=0,1,\ldots,s$ and $Q_{2i+1}(r;\mathcal{X})>0$ for each $i=0,1,\ldots,t$.
	
	From this point onward, the proof is essentially the same as the proof of \cite[Theorem 2.3]{NewtonLikeInequalities}. For this reason, we omit the details of the proof and instead refer the reader to \cite{NewtonLikeInequalities}. As in the cited proof, one can show that the polynomials
	\[
		f_0(w):=w^m-\binom{m}{1}Q_2(r;\mathcal{X})w^{m-1}+\cdots+(-1)^m\binom{m}{m}Q_{2m}(r;\mathcal{X})
	\]
	and
	\[
		f_1(w):=w^m-\binom{m}{1}\frac{Q_3(r;\mathcal{X})}{Q_1(r;\mathcal{X})}w^{m-1}+\cdots+(-1)^m\binom{m}{m}\frac{Q_{2m+1}(r;\mathcal{X})}{Q_1(r;\mathcal{X})}
	\]
	have real roots. Hence, we may apply Newton's Inequalities to the coefficients of $f_0$ and $f_1$, giving (\ref{eq:ASInequalitiesEven}) and (\ref{eq:ASInequalitiesOdd}), respectively.
\end{proof}

We are now ready to prove the main result of this paper:

\begin{proof}[Proof of Theorem \ref{thm:DiagonalElements}]
	Necessity was established in Observation \ref{obv:DiagonalElementsNecessity}, so now suppose (\ref{eq:s1ACondition}) and (\ref{eq:s2ACondition}) hold. We will show that $\sigma$ is the spectrum of a nonnegative matrix $A$ of the form (\ref{eq:AMatrix}).
	
	Let
	\[
		c_i:=e_i(-\lambda_2,-\lambda_3,\ldots,-\lambda_n) \hspace{2mm} : \hspace{6mm} i\in\mathbb{Z}
	\]
	and let
	\[
		g(x):=\prod_{i=2}^n(x-\lambda_i)=x^{n-1}+c_1x^{n-2}+c_2x^{n-3}+\cdots+c_{n-1}.
	\]
	Since $\mathrm{Re}(\lambda_i)\leq0$: $i=2,3,\ldots,n$, it follows that $c_i\geq0$ for all $i$. In addition, since $s_1(\Delta)=s_1(\sigma)=\rho-c_1$, we must have $\rho=c_1+s_1(\Delta)$, or, in other notation, $\rho=c_1+\epsilon_1^{(n)}$.
	
	Therefore, by Lemma \ref{lem:bjSol}, it sufices to show that the quantities $b_2,b_3,\ldots,b_n$ given in (\ref{eq:bjSol}) are nonnegative. Firstly, we note that (\ref{eq:Kij}) can be rearranged to give
	\[
		\mathscr{K}_{i,j}:=\eta_i^{(n-j+1)}c_1c_{j-i-1}+\left(\epsilon_1^{(n)}\eta_{i-1}^{(n-j+1)}-\eta_i^{(n-j+1)}\right)c_{j-i}.
	\]
	If $i>0$, then
	\[
		\eta_i^{(n-j+1)}=\sum_{r=1}^{n-j+1}a_r\eta_{i-1}^{(r)}\leq\sum_{r=1}^na_r\eta_{i-1}^{(n-j+1)}=\epsilon_1^{(n)}\eta_{i-1}^{(n-j+1)}.
	\]
	Hence
	\[
		\mathscr{K}_{i,j}\geq0 \hspace{2mm} : \hspace{6mm} i=1,2,\ldots,j; \hspace{4mm} j=2,3,\ldots,n.
	\]
	Note also that
	\[
		\mathscr{K}_{0,j}=c_1c_{j-1}-c_{j} \hspace{2mm} : \hspace{6mm} j=2,3,\ldots,n.
	\]
	By Theorem \ref{thm:DifferentParity}, $\mathscr{K}_{0,j}>0$ whenever $j$ is odd. Hence $b_j>0$ for each odd index $j$.
	
	Let us now consider $b_j$ for even indices $j$. By (\ref{eq:b2}), $b_2\geq0$. Now fix $k\in\{2,3,\ldots,\lfloor n/2 \rfloor\}$. We will show that $b_{2k}\geq0$. If $\mathscr{K}_{0,2k}\geq0$, we are done, so from now on, assume $\mathscr{K}_{0,2k}<0$. Now consider the polynomial
	\[
		q(x):=(x-c_1)g(x)=x^n-\mathscr{K}_{0,2}x^{n-2}-\mathscr{K}_{0,3}x^{n-3}-\cdots-\mathscr{K}_{0,n}.
	\]	
	By considering the even terms in $q(x)$, we see that Theorem \ref{thm:LSVariant} implies $\mathscr{K}_{0,2j}<0:$ $j=1,2,\ldots k$.
	
	At this point, it is helpful to consider the quantities
	\begin{align*}
		W_1&:=\mathscr{K}_{0,2k}\mathscr{K}_{1,2}-\mathscr{K}_{0,2}\mathscr{K}_{1,2k},\\
		W_2&:=\mathscr{K}_{0,2k}\mathscr{K}_{2,2}-\mathscr{K}_{0,2}\mathscr{K}_{2,2k}.
	\end{align*}
	We will show that $W_1,W_2\geq0$.
	
	Let us first examine $W_1$. Note that the expression for $\mathscr{K}_{1,2}$ simplifies to
	\[
		\mathscr{K}_{1,2}=\epsilon_1^{(n)}c_1.
	\]
	In addition, notice that $\mathscr{K}_{1,2k}$ may be written in the form
	\begin{equation}\label{eq:K1,2k}
		\mathscr{K}_{1,2k}=\epsilon_1^{(n-2k+1)}c_1c_{2k-2}+{\epsilon_1'}^{(n-2k+1)}c_{2k-1}.
	\end{equation}
	Since $c_1c_{2k-2}\geq c_{2k-1}$ by Theorem \ref{thm:DifferentParity}, it follows that
	\[
		\mathscr{K}_{1,2k}\geq \epsilon_1^{(n)}c_{2k-1}.
	\]
	Therefore
	\begin{align*}
		W_1&\geq \epsilon_1^{(n)}( \mathscr{K}_{0,2k}c_1-\mathscr{K}_{0,2}c_{2k-1})\\
		&=\epsilon_1^{(n)}(c_2c_{2k-1}-c_1c_{2k})\\
		&\geq0,
	\end{align*}
	where the final inequality is due to Theorem \ref{thm:DifferentParity}.
	
	Next we examine $W_2$. Similarly to (\ref{eq:K1,2k}), we note that
	\[
		\mathscr{K}_{1,2k-1}=\epsilon_1^{(n-2k+2)}c_1c_{2k-3}+{\epsilon_1'}^{(n-2k+2)}c_{2k-2}.
	\]
	Hence
	\[
		\mathscr{K}_{1,2k-1}+{\epsilon_1'}^{(n-2k+2)}\mathscr{K}_{0,2k-2}=\epsilon_1^{(n)}c_1c_{2k-3}\geq0,
	\]
	i.e.
	\begin{equation}\label{eq:DiagProof1}
		\mathscr{K}_{1,2k-1}\geq-{\epsilon_1'}^{(n-2k+2)}\mathscr{K}_{0,2k-2}.
	\end{equation}
	Letting $m=2$ and $j=2k$ in (\ref{eq:KijRecurrence}), we see that
	\begin{equation}\label{eq:DiagProof2}
		\mathscr{K}_{2,2k}=\epsilon_1^{(n-2k+1)}\mathscr{K}_{1,2k-1}-\epsilon_2^{(n-2k+2)}\mathscr{K}_{0,2k-2}
	\end{equation}
	and combining (\ref{eq:DiagProof2}) with (\ref{eq:DiagProof1}) gives
	\begin{equation}\label{eq:DiagProof3}
		\mathscr{K}_{2,2k}\geq-\left( \epsilon_1^{(n-2k+1)}{\epsilon_1'}^{(n-2k+2)}+\epsilon_2^{(n-2k+2)} \right)\mathscr{K}_{0,2k-2}.
	\end{equation}
	Next, we note that
	\begin{eqnarray*}
		\epsilon_2^{(n)}&=&\epsilon_2^{(n-2k+1)}+{\epsilon_2'}^{(n-2k+1)}+\epsilon_1^{(n-2k+1)}{\epsilon_1'}^{(n-2k+1)}\\
		&=&\left(\epsilon_2^{(n-2k+2)}-a_{n-2k+2}\epsilon_1^{(n-2k+1)}\right)+{\epsilon_2'}^{(n-2k+1)}\\
		&&+\epsilon_1^{(n-2k+1)}\left(a_{n-2k+2}+{\epsilon_1'}^{(n-2k+2)}\right)\\
		&=&\epsilon_2^{(n-2k+2)}+{\epsilon_2'}^{(n-2k+1)}+\epsilon_1^{(n-2k+1)}{\epsilon_1'}^{(n-2k+2)},
	\end{eqnarray*}
	that is to say,
	\[
		\epsilon_1^{(n-2k+1)}{\epsilon_1'}^{(n-2k+2)}+\epsilon_2^{(n-2k+2)}=\epsilon_2^{(n)}-{\epsilon_2'}^{(n-2k+1)}.
	\]
	Hence (\ref{eq:DiagProof3}) is equivalent to
	\[
		\mathscr{K}_{2,2k}\geq-\left( \epsilon_2^{(n)}-{\epsilon_2'}^{(n-2k+1)} \right)\mathscr{K}_{0,2k-2}.
	\]
	
	We also observe that, since $a_1\geq a_2\geq\cdots\geq a_n$,
	\[
		\binom{n}{2}^{-1}\epsilon_2^{(n)}\geq\binom{2k-1}{2}^{-1}{\epsilon_2'}^{(n-2k+1)}.
	\]
	Therefore
	\[
		{\epsilon_2'}^{(n-2k+1)}\leq\frac{\binom{2k-1}{2}}{\binom{n}{2}}\epsilon_2^{(n)}\leq\frac{\binom{2k-1}{2}}{\binom{2k}{2}}\epsilon_2^{(n)}=\left( 1-\frac{1}{k} \right)\epsilon_2^{(n)}
	\]
	and hence
	\begin{equation}\label{eq:DiagProof4}
		\mathscr{K}_{2,2k}\geq-\frac{1}{k}\epsilon_2^{(n)}\mathscr{K}_{0,2k-2}.
	\end{equation}
	
	Now note that
	\begin{align}\label{eq:DiagProof5}
		\mathscr{K}_{2,2}&=\epsilon_1^{(n)}\epsilon_1^{(n-1)}-\eta_2^{(n-1)}=\sum_{r=1}^{n-1}a_r\epsilon_1^{(n)}-\sum_{r=1}^{n-1}a_r\epsilon_1^{(r)}\notag\\
		&=\sum_{r=1}^{n-1}a_r{\epsilon_1'}^{(r)}=\epsilon_2^{(n)}.
	\end{align}
	By (\ref{eq:DiagProof4}) and (\ref{eq:DiagProof5}), it follows that
	\[
		W_2\geq\left( \mathscr{K}_{0,2k}+\frac{1}{k}\mathscr{K}_{0,2}\mathscr{K}_{0,2k-2} \right)\epsilon_2^{(n)};
	\]
	however, using the notation defined in (\ref{eq:QNotation}), we have
	\[
		-\mathscr{K}_{0,2i}=\binom{\lfloor n/2 \rfloor}{i}Q_{2i}(c_1;-\lambda_2,-\lambda_3,\ldots,-\lambda_n)\: : \hspace{2mm} i=1,2,\ldots,\lfloor n/2 \rfloor
	\]
	and hence, Lemma \ref{lem:ASInequalities} implies
	\[
		\frac{\mathscr{K}_{0,2}\mathscr{K}_{0,2k-2}}{-\mathscr{K}_{0,2k}}\geq\frac{\binom{\lfloor n/2 \rfloor}{1}\binom{\lfloor n/2 \rfloor}{k-1}}{\binom{\lfloor n/2 \rfloor}{k}}=k\frac{\lfloor n/2 \rfloor}{\lfloor n/2 \rfloor-k+1}>k.
	\]
	Hence $W_2\geq0$, as claimed.
	
	All that remains is to notice that
	\[
		\mathscr{K}_{0,2k}b_2-\mathscr{K}_{0,2}\left( \mathscr{K}_{0,2k}+\mathscr{K}_{1,2k}+\mathscr{K}_{2,2k} \right)=W_1+W_2\geq0
	\]
	and hence
	\[
		b_{2k}\geq\mathscr{K}_{0,2k}+\mathscr{K}_{1,2k}+\mathscr{K}_{2,2k}\geq\frac{\mathscr{K}_{0,2k}}{\mathscr{K}_{0,2}}b_2\geq0.\qedhere
	\]
\end{proof}

If $\sigma:=(\rho,\lambda_2,\lambda_3,\ldots,\lambda_n)$, where $\rho\geq0$ and $\mathrm{Re}(\lambda_i)\leq0$: $i=2,3,\ldots,n$, then we can use Theorem \ref{thm:DiagonalElements} to calculate the maximum possible diagonal element and the minimal possible diagonal element of a nonnegative matrix with spectrum $\sigma$:

\begin{corollary}\label{cor:MaxMinDiag}
	Let $\sigma:=(\rho,\lambda_2,\lambda_3,\ldots,\lambda_n)$ be realisable, where $\rho\geq0$ and $\mathrm{Re}(\lambda_i)\leq0$: $i=2,3,\ldots,n$. Then $\sigma$ is the spectrum of a nonnegative matrix with a diagonal element $a$ if and only if
	\begin{equation}\label{eq:DiagonalCor1}
		0\leq a\leq s_1(\sigma)
	\end{equation}
	and
	\begin{equation}\label{eq:DiagonalCor2}
		\left( a-\frac{s_1(\sigma)}{n} \right)^2\leq\frac{(n-1)(ns_2(\sigma)-s_1(\sigma)^2)}{n^2}.
	\end{equation}
\end{corollary}
\begin{proof}
	Let us write
	\[
		\delta:=\sqrt{(n-1)(ns_2(\sigma)-s_1(\sigma)^2)}.
	\]
	We note that since $\sigma$ is realisable, it must satisfy (\ref{eq:1stJLL}) and hence $\delta$ is real.
	
	The necessity of (\ref{eq:DiagonalCor1}) is obvious. To see that (\ref{eq:DiagonalCor2}) is necessary, suppose $\sigma$ is the spectrum of a nonnegative matrix with diagonal elements $(a_1,a_2,\ldots,a_n)$ and define
	\begin{equation}\label{eq:tiDefDE}
		t_i:=a_i-\frac{s_1(\sigma)}{n} \hspace{2mm} : \hspace{6mm} i=1,2,\ldots,n.
	\end{equation}
	In order to show that (\ref{eq:DiagonalCor2}) is necessary, we need to show that
	\[
		t_i^2\leq\frac{\delta^2}{n^2} \hspace{2mm} : \hspace{6mm} i=1,2,\ldots,n.
	\]
	In fact, since the $a_i$ are labelled arbitrarily, it suffices to show that $t_1^2\leq\delta^2/n^2$.
	
	Since $\sum_{i=1}^na_i=s_1(\sigma)$, it follows that $\sum_{i=1}^nt_i=0$. Hence, by the Cauchy-Schwarz inequality,
	\begin{equation}\label{eq:DiagonalCor3}
		\sum_{i=1}^nt_i^2=t_1^2+\sum_{i=2}^nt_i^2=\left( \sum_{i=2}^nt_i \right)^2+\sum_{i=2}^nt_i^2\geq\frac{n}{n-1}\left( \sum_{i=2}^nt_i \right)^2=\frac{nt_1^2}{n-1}.
	\end{equation}
	Combining (\ref{eq:DiagonalCor3}) with (\ref{eq:CloseWRTl2}) then gives
	\[
		\frac{nt_1^2}{n-1}\leq s_2(\sigma)-\frac{s_1(\sigma)^2}{n},
	\]
	or, equivalently, $t_1^2\leq\delta^2/n^2$, as required.
	
	Now suppose (\ref{eq:DiagonalCor1}) and (\ref{eq:DiagonalCor2}) hold and consider the list $\Delta=(a_1,a_2,$ $\ldots,a_n)$, where
	\begin{align*}
		a_1&=a,\\
		a_i&=\frac{s_1(\sigma)-a}{n-1} \hspace{2mm} : \hspace{6mm} i=2,3,\ldots,n.
	\end{align*}
	Define $t_i$ as before.
	
	By (\ref{eq:DiagonalCor1}), $a_i\geq0$ for all $i=1,2,\ldots,n$ and it is clear that $s_1(\Delta)=s_1(\sigma)$. Furthermore, we note that, since $a_2=a_3=\cdots=a_n$, we must have equality in (\ref{eq:DiagonalCor3}). Combine this with the fact that (\ref{eq:DiagonalCor2}) implies $t_1^2\leq\delta^2/n^2$ and we see that
	\[
		\sum_{i=1}^n\left( a_i-\frac{s_1(\sigma)}{n} \right)^2=\sum_{i=1}^nt_i^2=\frac{nt_1^2}{n-1}\leq\frac{\delta^2}{n(n-1)}=s_2(\sigma)-\frac{s_1(\sigma)^2}{n},
	\]
	which, as noted in Section \ref{sec:biComputation}, is equivalent to $s_2(\Delta)\leq s_2(\sigma)$. Therefore, by Theorem \ref{thm:DiagonalElements}, $\sigma$ is the spectrum of a nonnegative matrix with diagonal elements $\Delta$.
\end{proof}

\begin{remark}
	If there is equality in the JLL condition (\ref{eq:1stJLL}), then it follows from Corollary \ref{cor:MaxMinDiag} that  $\sigma$ can only be realised by a matrix with constant diagonal.
\end{remark}

We conclude this paper by giving an example which shows that (for the fixed structure given in (\ref{eq:AMatrix})) the ordering of the diagonal elements cannot be arbitrary.

\begin{example}
	Let $\sigma:=(4,i,-i,i,-i)$ and $\Delta:=(2,2,0,0,0)$. We have $s_1(\Delta)=s_1(\sigma)=4$ and $s_2(\Delta)=8<12=s_2(\sigma)$. Therefore, by Theorem \ref{thm:DiagonalElements}, $\sigma$ is the spectrum of a nonnegative matrix of the form (\ref{eq:AMatrix}), with the diagonal elements appearing in descending order. Computing the matrix entries $b_2,b_3,b_4,b_5$ from (\ref{eq:bjSol}) (or otherwise), we see that $\sigma$ is realised by the matrix
	\[
		A=\left[
\begin{array}{ccccc}
 2 & 1 & 0 & 0 & 0 \\
 0 & 2 & 1 & 0 & 0 \\
 0 & 0 & 0 & 1 & 0 \\
 0 & 0 & 0 & 0 & 1 \\
 50 & 55 & 16 & 2 & 0
\end{array}
\right].
	\]
	
	On the other hand, if we wish the diagonal elements to appear in ascending order, we are forced to choose
	\[
		A=\left[
\begin{array}{ccccc}
 0 & 1 & 0 & 0 & 0 \\
 0 & 0 & 1 & 0 & 0 \\
 0 & 0 & 0 & 1 & 0 \\
 0 & 0 & 0 & 2 & 1 \\
 4 & -1 & 8 & 2 & 2
\end{array}
\right],
	\]
	which is not nonnegative. Hence the requirement that $a_1\geq a_2\geq\cdots\geq a_n$ in Theorem \ref{thm:DiagonalElements} cannot be omitted.
\end{example}

\bibliographystyle{amsplain}
\bibliography{Bibliography}

\end{document}